\documentclass[a4paper]{amsart}
\usepackage{amssymb}
\usepackage{amsmath}
\usepackage{amsthm}
\usepackage{amsfonts}
\usepackage{hyperref}
\usepackage{tikz}
\usepackage{tikz-cd}

\newtheorem*{theorem*}{Theorem}

\newtheorem{theorem}{Theorem}[section]

\newtheorem{lemma}[theorem]{Lemma}
\newtheorem{question}{Question}

\numberwithin{equation}{section}

\renewcommand{\mod}{\operatorname{mod}}
\newcommand{\Ext}{\operatorname{Ext}}
\newcommand{\cx}{\operatorname{cx}}
\newcommand{\domdim}{\operatorname{domdim}}
\newcommand{\gldim}{\operatorname{gldim}}
\newcommand{\End}{\operatorname{End}}
\newcommand{\Hom}{\operatorname{Hom}}
\newcommand{\add}{\operatorname{\mathrm{add}}}

\newcommand{\K}{\mathbf{k}}
\begin{document}

\title{Existence of a 2-cluster tilting module does not imply finite complexity}
\date{\today}



\author{Ren\'{e} Marczinzik}
\address{Institute of algebra and number theory, University of Stuttgart, Pfaffenwaldring 57, 70569 Stuttgart, Germany}
\email{marczire@mathematik.uni-stuttgart.de}

\author{Laertis Vaso}
\address{Max Planck Institute for Mathematics, Vivatsgasse 7, 53111 Bonn, Germany}
\email{vaso@mpim-bonn.mpg.de}

\begin{abstract}
We give an example of a finite-dimensional algebra with a 2-cluster tilting module and a simple module which has infinite complexity. This answers a question of Erdmann and Holm.
\end{abstract}

\maketitle

\section*{Introduction}

In \cite{Iya}, Iyama generalised the classical correspondence between representation-finite algebras and Auslander algebras due to Auslander \cite{Aus}, see also \cite[Chapter VI.5.]{ARS}. More specifically Iyama established that finite-dimensional algebras admitting an $n$-cluster tilting module are in bijective correspondence with higher Auslander algebras. Moreover, $2$-cluster tilting modules are of special importance as they have several applications to cluster algebras via preprojective algebras (e.g. \cite{GLS}) and connections with Jacobian algebras of quivers with potential (e.g. \cite{HI}).

In general it is not easy to find $n$-cluster tilting modules. In \cite{EH} the authors show that selfinjective algebras which admit $n$-cluster tilting modules are particularly rare. More specifically, they showed that if a selfinjective algebra $A$ admits an $n$-cluster tilting module, then all $A$-modules have complexity at most $1$. It immediately follows from this that the terms in a minimal projective resolution of any $A$-module have bounded dimensions. Thus the existence of an $n$-cluster tilting module for selfinjective algebras gives global information on the whole module category. In \cite[Section 5.5]{EH} the authors posed the question whether this result holds more generally for all algebras:

\begin{question} \label{main question EH}
Let $A$ be a connected finite-dimensional algebra admitting an $n$-cluster tilting module for some $n \geq 2$. Does every $A$-module have complexity at most one?
\end{question}

Even though it is not easy to find $n$-cluster tilting modules, there are many examples for certain classes of algebras. The case of algebras with finite global dimension has attracted a lot of attention (e.g. \cite{HI, IO, Vas, CIM, Vas2}); however in this case all modules have trivially complexity equal to zero. Another well studied case is that of selfinjective algebras (e.g. \cite{EH, DI, CDIM}), where the complexity of all modules is at most one when there exists an $n$-cluster tilting module by the main result of \cite{EH}. Thus a negative answer to Question \ref{main question EH} needs to involve an algebra which is not selfinjective, has infinite global dimension and admits an $n$-cluster tilting module. 

In this article we answer Question \ref{main question EH}. In particular, the following is our main theorem which gives a negative answer to the question by Erdmann and Holm.

\begin{theorem*}
There exists a connected algebra $A$ that admits a $2$-cluster tilting module and a simple $A$-module $S$ which has infinite complexity.
\end{theorem*}

\section{An algebra with a 2-cluster tilting module and a simple module with infinite complexity}

\subsection{Preliminaries}

Let $\K$ be a field. In this article by algebra we mean finite-dimensional $\K$-algebra and by module we mean finite-dimensional right module. We also assume that all algebras are connected. We assume that the reader is familiar with the basics of representation theory and homological algebra of finite-dimensional algebras; we refer for example to the textbooks \cite{ARS, ASS, SY} for an introduction.

All subcategories considered are closed under isomorphisms and $\mathbb{N}_0$ denotes the natural numbers including zero.

Let $A$ be an algebra. We denote by $\mod-A$ the category of finite-dimensional $A$-modules and by $D=\Hom_K(-,K)$ the natural duality on the module category $\mod-A$. We denote by $\tau$ and $\tau^{-}$ the Auslander--Reiten translations. For an $A$-module $M$ we denote by $\add(M)$ the \emph{additive closure of $M$}, that is the full subcategory of $\mod-A$ consisting of direct summands of $M^n$ for some $n \geq 1$.  An $A$-module $M$ is called an \emph{$n$-cluster tilting module} if 
\begin{align*}
    \add(M) &= \{X \in \mod-A \mid \Ext_A^i(M,X)=0 \ \text{for} \ 1 \leq i \leq n-1 \} \\
    &= \{X \in \mod-A \mid \Ext_A^i(X,M)=0  \ \text{for} \ 1 \leq i \leq n-1 \}
\end{align*}
Notice that in some references (e.g. \cite{Iya4, EH}) $n$-cluster tilting module are also called \emph{maximal $(n-1)$-orthogonal modules for $n \geq 2$}.

Let 
\[0 \rightarrow A \rightarrow I^0 \rightarrow I^1 \rightarrow \cdots \]
be a minimal injective coresolution of the regular module $A$. The \emph{dominant dimension} $\domdim A$ is defined as the smallest $n \geq 0$ such that $I^n$ is not projective. The \emph{global dimension} $\gldim A$ is defined as the supremum of the projective dimensions of all $A$-modules. The algebra $A$ is called a \emph{higher Auslander algebra} if $\gldim A = \domdim A$ and $\gldim A \geq 2$.

The following theorem due to Iyama (see for example \cite[Theorem 2.6]{Iya2} for a quick proof) gives a fundamental connection between $n$-cluster tilting module and higher Auslander algebras.

\begin{theorem} \label{Iyamaresult}
Let $A$ be a finite-dimensional algebra. Then $M \in \mod-A$ is an $n$-cluster tilting module if and only if the algebra $\End_A(M)$ is a higher Auslander algebra of global dimension $n+1$. 
\end{theorem}

Recall that the \emph{complexity} $\cx(M)$ of a module $M$ with minimal projective resolution 
\[\cdots \rightarrow P_n \rightarrow P_{n-1} \rightarrow \cdots \rightarrow P_1 \rightarrow P_0 \rightarrow M \rightarrow 0\]
is defined as 
\[\cx(M)= \inf \{ b \in \mathbb{N}_0 \mid \exists\; c >0: \dim P_n \leq c n^{b-1} \ \text{for all} \ n \}.\]
Thus the complexity of a module $M$ is at most one if and only if the terms $P_n$ of a projective resolution of $M$ have bounded vector space dimensions.
When no $b \in \mathbb{N}_0$ exists with $\dim P_n \leq c n^{b-1} \ \text{for all} \ n$ for some $c>0$ then the complexity of a module $M$ is infinite.

\subsection{Main result}

For the rest of this section, let $Q$ be the quiver
\[
\begin{tikzcd}
     1 \arrow[bend left=30, r] & 2 \arrow[bend left=30, l] \arrow[bend left=30, r] & 3, \arrow[bend left=30, l]
\end{tikzcd}
\]
and let $A=\K Q/ J^2$ where $J$ is the ideal of $\K Q$ generated by the arrows. We denote by $S_i$ the simple $A$-module corresponding to the vertex $i\in Q_0$. 

\begin{lemma}\label{lem:2-ct}
The module $M:=A\oplus D(A)$ is a $2$-cluster tilting module.
\end{lemma}

\begin{proof}
We need to show that 
\begin{align*}
    \add(M) &= \{X \in \mod-A \mid \Ext_A^1(M,X)=0\} \\
    &= \{X \in \mod-A \mid \Ext_A^1(X,M)=0 \}
\end{align*}
It is enough to show the first equality; the second follows dually by the symmetry of $Q$.

The Auslander--Reiten quiver $\Gamma(A)$ of $A$ is 
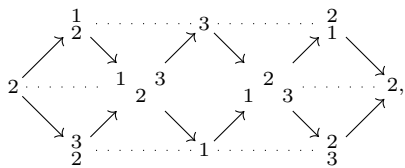
\[\begin{tikzpicture}[scale=1.2, transform shape]
\tikzstyle{nct3}=[circle, minimum width=6pt, draw=none, inner sep=0pt, scale=0.9]

\node[nct3] (A) at (0,0) {$\begin{smallmatrix} 2 \end{smallmatrix}$};
\node[nct3] (B) at (0.7,0.7) {$\begin{smallmatrix} 1 \\ 2 \end{smallmatrix}$};
\node[nct3] (C) at (0.7,-0.7) {$\begin{smallmatrix} 3 \\ 2 \end{smallmatrix}$};
\node[nct3] (D) at (1.4,0) {$\begin{smallmatrix} 1 && 3 \\ & 2 &
\end{smallmatrix}$};
\node[nct3] (E) at (2.1,0.7) {$\begin{smallmatrix} 3 \end{smallmatrix}$};
\node[nct3] (F) at (2.1,-0.7) {$\begin{smallmatrix} 1 \end{smallmatrix}$};
\node[nct3] (G) at (2.8,0) {$\begin{smallmatrix} & 2 & \\ 1 && 3
\end{smallmatrix}$};
\node[nct3] (H) at (3.5,0.7) {$\begin{smallmatrix} 2 \\ 1 \end{smallmatrix}$};
\node[nct3] (I) at (3.5,-0.7) {$\begin{smallmatrix} 2 \\ 3 \end{smallmatrix}$};
\node[nct3] (J) at (4.2,0) {$\begin{smallmatrix} 2, \end{smallmatrix}$};

\draw[->] (A) to (B);
\draw[->] (A) to (C);
\draw[->] (B) to (D);
\draw[->] (C) to (D);
\draw[->] (D) to (E);
\draw[->] (D) to (F);
\draw[->] (E) to (G);
\draw[->] (F) to (G);
\draw[->] (G) to (H);
\draw[->] (G) to (I);
\draw[->] (H) to (J);
\draw[->] (I) to (J);

\draw[loosely dotted] (A.east) -- (D);
\draw[loosely dotted] (B.east) -- (E);
\draw[loosely dotted] (C.east) -- (F);
\draw[loosely dotted] (E.east) -- (H);
\draw[loosely dotted] (F.east) -- (I);
\draw[loosely dotted] (G.east) -- (J);
\end{tikzpicture}\]
where modules are denoted using their composition series. Notice in particular that an indecomposable module is either simple or a direct summand of $M$.

In the rest of this proof we denote by $P_i$ the projective cover of $S_i$ and by $I_i$ the injective envelope of $S_i$.

To show that $\add(M) \subseteq \{X \in \mod-A \mid \Ext_A^1(M,X)=0\}$ it is enough to show that $\Ext_A^1(M,M)=0$. Since $M=A\oplus D(A)$, it is enough to show that $\Ext_A^1(D(A),A)=0$. Using the Auslander--Reiten formula \cite[Chapter IV, Theorem 2.13]{ASS}, we have
\begin{align*}
    \Ext_A^1(D(A),A) &= D\underline{\Hom}_A(\tau^{-}(A), D(A)) = D\underline{\Hom}_A(S_1\oplus S_2\oplus S_3, D(A)) = 0,
\end{align*}
where the last equality comes from the fact that 
\begin{align*}
     &\Hom_{A}(S_1,P_2) \twoheadrightarrow \Hom_{A}(S_1,I_1), &&\Hom_{A}(S_1, I_2) = 0, &&\Hom_{A}(S_1, I_3)=0,\\ 
     &\Hom_{A}(S_2, I_1) = 0, &&\Hom_{A}(S_2,P_1) \twoheadrightarrow \Hom_{A}(S_2,I_2), &&\Hom_{A}(S_2,I_3)=0,\\
    &\Hom_{A}(S_3,I_1)=0, &&\Hom_{A}(S_3, I_2) = 0,  &&\Hom_{A}(S_3,P_2) \twoheadrightarrow \Hom_{A}(S_3, I_3), 
\end{align*}
which can be immediately verified by looking at $\Gamma(A)$.

It remains to show the inclusion $\{X \in \mod-A \mid \Ext_A^1(M,X)=0\}\subseteq \add(M)$. Let $X\in\mod-A$ be such that $\Ext_A^1(M,X)=0$. By additivity of $\Ext_A^1(M,-)$, we may assume that $X$ is indecomposable. Since $\tau(I_1)=S_3$, $\tau(I_2)=S_2$ and $\tau(I_3)=S_1$, it follows that $\Ext_A^1(M,S_i)\neq 0$ for $i\in\{1,2,3\}$. Hence $X$ is not simple and so $X$ is a direct summand of $M$, which completes the proof.
\end{proof}

Alternatively we can also calculate quiver and relations of the algebra $B=\End_A(M)$ to see that $B$ is a higher Auslander algebra of global dimension $3$ and thus $M$ is a $2$-cluster tilting module by Theorem \ref{Iyamaresult}. This can be verified by a direct computation or  by using for example the GAP-package \cite{QPA}. For convenience of the reader, we give a presentation of $B$ by a quiver with relations. If $Q_B$ is the quiver
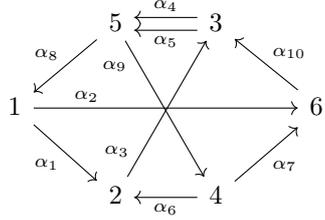
\[
\begin{tikzcd}
    {} & 5 \arrow[dl, swap, "\alpha_8"] \arrow[ddr, swap, pos=0.1, "\alpha_9"] & 3 \arrow[l, swap, shift right, "\alpha_4"] \arrow[l, shift left, "\alpha_5"] & {} \\
    1 \arrow[rrr, pos=0.2, "\alpha_2"] \arrow[dr, swap, "\alpha_1"] & {} & {} & 6, \arrow[ul, swap, "\alpha_{10}"] \\
    {} & 2 \arrow[uur, pos=0.1, "\alpha_3"] & 4 \arrow[l, "\alpha_6"] \arrow[ur, swap, "\alpha_7"] & {}
\end{tikzcd}
\]
and $J_B$ is the ideal of $\K Q_B$ given by
\begin{align*}
    J_B = \langle &\alpha_8\alpha_1-\alpha_9\alpha_6,\;\; \alpha_8\alpha_2-\alpha_9\alpha_7,\;\; \alpha_6\alpha_3-\alpha_7\alpha_{10},\;\; \alpha_1\alpha_3-\alpha_2\alpha_{10}, \\
    {} &\alpha_5\alpha_8\alpha_1,\; \alpha_5\alpha_8\alpha_2,\; \alpha_1\alpha_3\alpha_5,\; \alpha_6\alpha_3\alpha_5,\; \alpha_3\alpha_4,\; \alpha_4\alpha_8,\; \alpha_{10}\alpha_5,\; \alpha_5\alpha_9 \rangle,
\end{align*} 
then $B\cong \K Q_B/J_B$.

\begin{lemma}\label{lem:infinite complexity}
The simple $A$-module $S_2$ has infinite complexity.
\end{lemma}

\begin{proof}
A direct computation shows that
\[\Omega(S_2) = S_1 \oplus S_3,\;\; \Omega^2(S_2) = S_2 \oplus S_2.\]
A straightforward induction on $n\geq 0$ then shows that
\begin{equation}\label{syzygies of S2}
    \Omega^n(S_2) = \begin{cases} S_1^{\oplus 2^{\frac{n-1}{2}}}\oplus S_3^{\oplus 2^{\frac{n-1}{2}}}, &\mbox{if $n$ is odd,} \\ S_2^{\oplus 2^{\frac{n}{2}}},  &\mbox{if $n$ is even.} \end{cases}    
\end{equation} 
Let
\[\cdots \rightarrow P_n \rightarrow P_{n-1} \rightarrow \cdots \rightarrow P_1 \rightarrow P_0 \rightarrow S_2 \rightarrow 0\]
be a minimal projective resolution of $S_2$. Since $\Omega^n(S_2)$ is semisimple for any $n\geq 0$, it follows that the number of direct summands of $P_n$ is at least equal to the number of direct summands of $\Omega^n(S_2)$. By (\ref{syzygies of S2}) we have that $\Omega^n(S_2)$ has at least $2^{\frac{n}{2}}$ direct summands and so $\dim(P_n)\geq 2^{\frac{n}{2}}$. 

Now assume towards a contradiction that there exists a $b\in \mathbb{N}_0$ such that there exists a $c>0$ with $\dim P_n\leq cn^{b-1}$ for all $n\geq 0$. Then
\begin{align*}
    2^{\frac{n}{2}} \leq \dim P_n \leq cn^{b-1},\;\;\; &\text{ for all $n\geq 0$}\\
    \intertext{implies that}
    \frac{2^{\frac{n}{2}}}{n^{b-1}}\leq c,\;\;\; &\text{ for all $n\geq 1$},
\end{align*}
which is a contradiction, since $\lim\limits_{n\to\infty}\frac{2^{\frac{n}{2}}}{n^{b-1}}=\infty$. Hence no such $b$ exists and $\cx(S_2)=\infty$ which finishes the proof.
\end{proof}

With this we are ready to give our main result.

\begin{theorem}
The algebra $A$ has a $2$-cluster tilting module and there exists a simple $A$-module with infinite complexity.
\end{theorem}

\begin{proof}
Follows immediately by Lemma \ref{lem:2-ct} and Lemma \ref{lem:infinite complexity}.
\end{proof}

\section*{Acknowledgments} 
Ren{\'e} Marczinzik is funded by the DFG with the project number 428999796. Laertis Vaso is grateful to Max Planck Institute for Mathematics in Bonn for its hospitality and financial support. We profited from the use of the GAP-package \cite{QPA}.

\end{document}